\documentclass[11pt]{amsart}

\usepackage{amsmath}
\usepackage{amsfonts}
\usepackage{amssymb}

\usepackage{hyperref}
\hypersetup{colorlinks,linkcolor={red},citecolor={blue},urlcolor={blue}}

\evensidemargin\oddsidemargin

\newtheorem{theorem}{Theorem}[section]
\newtheorem{lemma}[theorem]{Lemma}
\newtheorem{corollary}[theorem]{Corollary}

\theoremstyle{definition}
\newtheorem{definition}[theorem]{Definition}

\theoremstyle{remark}
\newtheorem{remark}[theorem]{Remark}
\newtheorem{question}[theorem]{Question}

\numberwithin{equation}{section}

\newcommand{\CC}{\mathbb C}
\newcommand{\HH}{\mathbb H}

\newcommand{\NN}{\mathbb N}
\newcommand{\PP}{\mathbb P}
\newcommand{\QQ}{\mathbb Q}

\newcommand{\ZZ}{\mathbb Z}

\newcommand{\cD}{\mathcal D}
\newcommand{\cH}{\mathcal H}

\newcommand{\SL}{\mathop{\mathrm {SL}}\nolimits}

\newcommand{\Sp}{\mathop{\mathrm {Sp}}\nolimits}
\newcommand{\Orth}{\mathop{\null\mathrm {O}}\nolimits}

\newcommand{\sign}{\mathop{\mathrm {sign}}\nolimits}

\newcommand{\Mp}{\mathop{\mathrm {Mp}}\nolimits}
\newcommand{\rank}{\mathop{\mathrm {rank}}\nolimits}

\newcommand{\latt}[1]{{\langle{#1}\rangle}}
\newcommand{\Kthree}{\mathop{\mathrm {K3}}\nolimits}

\newcommand{\Grit}{\operatorname{Grit}}
\newcommand{\Borch}{\operatorname{Borch}}
\def\div{\operatorname{div}}
\newcommand{\m}{\operatorname{mod}}
\newcommand{\im}{\operatorname{Im}}

\newcommand{\mult}{\operatorname{mult}}
\newcommand{\norm}{\operatorname{Norm}}

\newcommand{\Sch}{\operatorname{Sch}}

\newcommand{\II}{\operatorname{II}}
\newcommand{\ord}{\operatorname{ord}}
\newcommand{\abs}[1]{\lvert#1\rvert}

\begin{document}

\title[Theta block conjecture]{Theta block conjecture for Siegel paramodular forms}

\author{Haowu Wang}

\address{Max-Planck-Institut f\"{u}r Mathematik, Vivatsgasse 7, 53111 Bonn, Germany}

\email{haowu.wangmath@gmail.com}


\date{\today}


\begin{abstract}
The theta-block conjecture proposed by Gritsenko--Poor--Yuen in 2013 characterizes Siegel paramodular forms which are simultaneously Borcherds products and additive Jacobi lifts. In this paper, we prove this conjecture for two new infinite series of theta blocks of weights 2 and 3. The proof is based on Scheithauer's classification of reflective modular forms of singular weight. 
\end{abstract}

\maketitle

\section{Introduction}
Let $N$ be a positive integer. A Siegel paramodular form of level $N$ is a Siegel modular form of degree two with respect to the paramodular group of level $N$ defined as (see \cite{GN98})
$$
\Gamma_N=\left(\begin{array}{cccc}
  * & N* & * & * \\ 
  * & * & * & */N \\ 
  * & N* & * & * \\ 
  N* & N* & N* & *
  \end{array}   \right) \cap \Sp_2(\QQ),\quad \text{\rm all}\ *\in \ZZ.
$$
The Fourier-Jacobi coefficients of a Siegel paramodular form are holomorphic Jacobi forms in the sense of Eichler-Zagier \cite{EZ85}. Conversely, one can construct paramodular forms from Jacobi forms. The first method is the additive Jacobi lifting due to Gritsenko (see \cite{Gri94}) which sends a holomorphic Jacobi form to a paramodular form. The second method is a variant of Borcherds product proposed by Gritsenko--Nikulin (see \cite{Bor98,GN98}), which lifts a weakly holomorphic Jacobi form of weight $0$ to a meromorphic paramodular form with known divisors. In some sense, the additive Jacobi lifting (Gritsenko lift) is like an infinite sum and the Borcherds product has an infinite product expansion. In \cite{GPY15}, V. Gritsenko, C. Poor and D. Yuen investigated the paramodular forms which are simultaneously Borcherds 
products and Gritsenko lifts. One sees from their shapes that if the Gritsenko lift $\Grit(\phi)$ is a Borcherds product then the holomorphic Jacobi form $\phi$ must be a theta block defined below. Moreover, $\phi$ has vanishing order one in $q=e^{2\pi i \tau}$ (see \cite{PSY18}). Furthermore, it was conjectured in \cite{GW18} that $\phi$ is a pure theta block.

Let $\NN$ be the set of nonnegative integers and $f:\NN \to \ZZ$ be a function with a finite support.
A {\it theta block} (see \cite{GSZ19} for a full theory) is a function of the form
\begin{equation}\label{thblock}
\Theta_f(\tau,z)=\eta^{f(0)}(\tau)
\prod_{a=1}^{\infty}(\vartheta_a(\tau,z)/\eta(\tau))^{f(a)},
\end{equation}
where $\eta(\tau)=q^{\frac{1}{24}}\prod_{n\ge 1}(1-q^n)$ is the Dedekind eta-function, $\vartheta_a(\tau,z)=\vartheta(\tau,az)$ and 
$$
\vartheta(\tau,z)= q^{\frac{1}{8}}
(\zeta^{\frac{1}{2}}-\zeta^{-\frac{1}{2}}) \prod_{n\geq 1} (1-q^{n}\zeta)
(1-q^n \zeta^{-1})(1-q^n), \quad \zeta^r=e^{2\pi i rz}
$$
is the odd Jacobi theta-series which is a holomorphic Jacobi form of weight $\frac{1}{2}$ and index $\frac{1}{2}$ with a multiplier system of order $8$ (see \cite{GN98}). The function $\Theta_f$ is called a \textit{pure theta block} if $f$ is nonnegative on $\NN$. The pure theta block $\Theta_f$ is just a weak Jacobi form in general, but it will be a holomorphic Jacobi form for some good $f$. This gives a great way to construct holomorphic Jacobi forms of small weight explicitly.

In \cite{GPY15}, V. Gritsenko, C. Poor and D. Yuen formulated the following conjecture which gives a sufficient condition for a Gritsenko lift being a Borcherds product.
\smallskip

\noindent
{\bf Theta block conjecture.} 
{\it Let the pure theta block $\Theta_f$ be a holomorphic Jacobi form of weight 
$k$ and index $N$ with vanishing order one in $q=e^{2\pi i\tau}$. 
We define a weak Jacobi form $\Psi_f=-(\Theta_f\lvert T_{-}(2))/\Theta_f$ of weight $0$ and index $N$, where $T_{-}(2)$ is the index raising Hecke operator. Then
$$
\Grit(\Theta_f)=\Borch(\Psi_f)
$$
is a holomorphic paramodular form of weight $k$ with respect to $\Gamma_N$.}
\smallskip

It is not clear how difficult the conjecture will turn out to be, but there is no obvious way to attack it.  A natural idea is to consider the quotient $\Grit(\Theta_f) / \Borch(\Psi_f)$. If one can prove $\div(\Grit(\Theta_f))\supset \div(\Borch(\Psi_f))$, then $\Grit(\Theta_f)/\Borch(\Psi)$ will be a holomorphic modular form of weight zero and then equals a constant by K\"ocher's principle. Unfortunately, it is very hard to prove this because the divisor of $\Borch(\Psi)$ is usually very complicated. To pass through  this difficulty, one lifts $\Theta_f$ to a Jacobi form $\Theta_L$ in many variables associated to 
a certain positive definite lattice $L$ (i.e. Jacobi forms of lattice index, see \cite{Gri94, CG13}). Then we study the same question in the context of modular forms on  orthogonal groups of signature $(2,n)$. Notice that paramodular forms can be realized as orthogonal modular forms of signature $(2,3)$ (see \cite{GN98}).  Since Jacobi forms in many variables have stronger symmetry, the corresponding Borcherds product may have much simpler divisor such that it is easy to prove the including relation between divisors of Gritsenko lift and Borcherds product. Then one can prove the desired identity by considering the specializations of the identity in many variables.  When the lattice $L$ satisfies the $\norm_2$ condition defined in \cite[(3.8)]{GW18}, the above algorithm works well because the divisor of Borcherds product is determined completely by the divisor of theta block in this case (see \cite{GW18} for more details). Following this idea, the conjecture has been proved for all theta blocks of weights larger than $3$ (see \cite[\S 8]{GPY15}), for an infinite series of theta blocks of weight $3$ related to the root system $3A_2$ (see \cite{Gri18, GSZ19}), and for an infinite series of theta blocks of weight $2$ related to the root system $A_4$ (see \cite{GW17, GW18}).

In this paper, we prove the conjecture for an infinite series of theta blocks of weight 2 related to the root system $A_1\oplus B_3$, and for an infinite series of theta blocks of weight 3 related to the root system $2A_1\oplus B_2\oplus A_2$. 

\begin{theorem}\label{th:wt2}
Let $\mathbf{a}=(a_1, a_2, a_3, a_4)\in \ZZ^4$. Then we have
$$
\varphi_{2,\mathbf{a}}=\eta^{-6}
\vartheta_{2a_1+a_2+a_4}\vartheta_{a_2}\vartheta_{a_2+a_3}\vartheta_{a_2+2a_3+2a_4}
\vartheta_{a_2+a_3+a_4}\vartheta_{a_2+a_3+2a_4}\vartheta_{a_3}
\vartheta_{a_3+a_4}\vartheta_{a_3+2a_4}\vartheta_{a_4}
\in J_{2,N(\mathbf{a})},
$$
where
$$
N(\mathbf{a})=2a_1^2 + 2a_1a_2+2a_1a_4+3a_2^2+5a_2a_3+6a_2a_4
+5a_3^2+10a_3a_4+8a_4^2.
$$
Moreover, for any  $\mathbf{a}\in \ZZ^4$ such that $\varphi_{2,\mathbf{a}}$ is not 
identically zero, one has
$$\Grit(\varphi_{2,\mathbf{a}})= 
\Borch\left(-\frac{\varphi_{2,\mathbf{a}}|T_{-}(2)}{\varphi_{2,\mathbf{a}}}
\right)\in M_2(\Gamma_{N(\mathbf{a})}).
$$
\end{theorem}

\begin{theorem}\label{th:wt3}
Let $\mathbf{b}=(b_1, b_2, b_3, b_4, b_5, b_6)\in \ZZ^6$. Then we have
$$
\varphi_{3,\mathbf{b}}=\eta^{-3}
\vartheta_{b_1}\vartheta_{2b_2+b_3-b_1}\vartheta_{b_3}\vartheta_{b_3+b_4}
\vartheta_{b_3+2b_4}\vartheta_{b_4}\vartheta_{b_5}
\vartheta_{b_6}\vartheta_{b_5+b_6}
\in J_{3,N(\mathbf{b})},
$$
where
$$
N(\mathbf{b})=b_1^2-2b_1b_2-b_1b_3+2b_2^2+2b_2b_3+2b_3^2+3b_3b_4+3b_4^2+b_5^2+b_5b_6+b_6^2.
$$
Moreover, for any  $\mathbf{b}\in \ZZ^6$ such that $\varphi_{3,\mathbf{b}}$ is not 
identically zero, one has
$$\Grit(\varphi_{3,\mathbf{b}})= 
\Borch\left(-\frac{\varphi_{3,\mathbf{b}}|T_{-}(2)}{\varphi_{3,\mathbf{b}}}
\right)\in M_3(\Gamma_{N(\mathbf{b})}).
$$
\end{theorem}

There are two new difficulties to prove the above theorems in comparison with the previously proved cases. We next explain the detail in the case of weight 2. The case of weight 3 is similar.

(I) There are exactly four infinite series of theta blocks of weight 2 in \cite{GSZ19}. They are related to the root systems $A_4$, $B_2\oplus G_2$, $A_1\oplus B_3$ and $A_1\oplus C_3$, respectively. In order to prove the theta block conjecture for these infinite series, we need to lift them to suitable Jacobi forms of lattice index and lift the corresponding paramodular forms to some orthogonal modular forms. Thus it is  crucial to find the proper lattices.
The case of $A_4$ was solved in \cite{GW17, GW18}. From the shape of the theta block, we saw that $A_4^\vee(5)$ is the satisfying lattice. But for the other three infinite series,  the choice of the lattice is not natural at all because the root system is reducible and its symmetry is not strong enough. This is the first difficulty.  We next explain the strategy to overcome it. In general, there will be much stronger symmetry if modular forms have more variables. Therefore, it will be better to lift the paramodular forms to some orthogonal modular forms of singular weight. There is a conjecture based on experience that every Borcherds product of singular weight comes from a reflective modular form  (see \cite{Wan19}). Here, reflective modular forms are modular forms on orthogonal groups whose divisors are determined by reflections in the orthogonal group (see \cite{GN98, Gri18}).  In \cite{Sch06, Sch17}, Scheithauer classified reflective modular forms of singular weight on lattices of squarefree level. In view of these facts, we search for reflective modular forms of singular weight 2 on lattices containing two hyperbolic planes in Scheithauer's list. Finally, we found only one new such modular form and its pullbacks give the infinite series appearing in Theorem \ref{th:wt2}. We do not know what is the proper lifting for the rest two infinite series. 

(II) The second difficulty of the proof is that the corresponding lattice does not satisfy the condition $\norm_2$. Therefore we can not employ the approach in \cite{GW18} to show that the associated reflective modular form is an additive lifting. We explain how to surmount this difficulty. Firstly, we show that the additive lifting vanishes on one part of reflective divisors. Secondly, we prove that the input of additive lifting (as a vector-valued modular form) is invariant up to a character under the orthogonal group of the discriminant group of the lattice.  From the reconstruction of additive lifts due to Borcherds in terms of vector-valued modular forms, we conclude that the additive lifting is a modular form for the full orthogonal group. Combining the two facts together, it is not hard to prove the pivotal result that the additive lifting vanishes on all reflective divisors because the level of the lattice is squarefree.

The layout of this paper is as follows. In the next section, we introduce briefly reflective modular forms and Jacobi forms of lattice index. Sections 3 and 4 are devoted to the proofs of our theorems.

\section{Reflective modular forms and Jacobi forms of lattice index}
As mentioned in the introduction, we will prove the theorems in the context of orthogonal modular forms. In this section, we introduce the necessary materials.

We first fix some notations. For an even lattice $M$, we denote its dual lattice by $M^\vee$ and its discriminant form by $D(M)=M^\vee/M$. The level of $M$ is the minimal positive integer $N$ such that $N(x,x)\in 2\ZZ$ for all $x\in M^\vee$. For $v\in M^\vee$, we denote the order of $v$ in $D(M)$ by $\ord(v)$. For any integer $a$, the lattice obtained by rescaling $M$ with $a$ is denoted by $M(a)$.  

Let $M$ be an even lattice of signature $(2,n)$ with $n\geq 3$. The Hermitian symmetric domain of type IV associated to $M$ is defined as (we choose one of the two connected components)
\begin{equation*}
\cD(M)=\{[\omega] \in  \PP(M\otimes \CC):  (\omega, \omega)=0, (\omega,\bar{\omega}) > 0\}^{+}.
\end{equation*}
Let $\Orth^+(M) \subset \Orth(M)$ be the index $2$ subgroup preserving $\cD(M)$. The stable orthogonal group $\widetilde{\Orth}^+(M)$ is the kernel of the natural homomorphism $\Orth^+ (M) \to \Orth(D(M))$. Let $\Gamma$ be a finite index subgroup of $\Orth^+ (M)$ and $k\in \NN$. A modular form of weight $k$ and character $\chi: \Gamma\to \CC^*$ for $\Gamma$ is a holomorphic function $F: \cD(M)^{\bullet}\to \CC$ on the affine cone $\cD(M)^{\bullet}$ satisfying
\begin{align*}
F(t\mathcal{Z})&=t^{-k}F(\mathcal{Z}), \quad \forall t \in \CC^*,\\
F(g\mathcal{Z})&=\chi(g)F(\mathcal{Z}), \quad \forall g\in \Gamma.
\end{align*}

By \cite{Bor95}, $F$ either has weight 0 in which case it is constant, or has weight at least $n/2-1$. The minimal possible weight $n/2-1$ is called the singular weight.

For any $v\in M^\vee$ satisfying $(v,v)<0$, we define the rational quadratic divisor associated to $v$ as
$$
 \cD_v(M)=v^\perp\cap \cD(M)=\{ [Z]\in \cD(M) : (Z,v)=0\}. 
$$

A modular form $F$ is called reflective if its divisor is a union of reflective divisors $\cD_r(M)$, here $\cD_r(M)$ is called reflective if $r$ is reflective, namely $r\in M^\vee$ is primitive and the reflection $\sigma_r: x \mapsto x-\frac{2(r,x)}{(r,r)}r$ belongs to $\Orth^+(M)$. When the level of $M$ is a squarefree number $N$, a primitive vector $r\in M^\vee$ with $(r,r)<0$ is reflective if and only if there exists a positive integer $d \vert N$ such that $(r,r)=-\frac{2}{d}$ and $r$ has order $d$ in $D(M)$ (see \cite[Proposition 2.5]{Sch06}). In the case, $\cD_r$ is called a $2d$-reflective divisor. 

Reflective modular forms were first introduced by Borcherds \cite{Bor98} and Gritsenko-Nikulin \cite{GN98}, and they have applications in algebra and geometry (see e.g. a survey \cite{Gri18}). The number of such modular forms was conjectured to be finite by Gritsenko-Nikulin \cite{GN98}. Some classification results can be found in \cite{Sch06, Sch17, Ma17, Ma18a, Dit18, Wan18, Wan19}.

We next define Jacobi forms of lattice index as Fourier-Jacobi coefficients of orthogonal modular forms (see \cite{CG13} for details).
Let us assume that $M$ contains two hyperbolic planes, i.e.
$M=U\oplus U_1\oplus L(-1)$,
where $U=\ZZ e\oplus\ZZ f$ ($(e,e)=(f,f)=0$, $(e,f)=1$),
$U_1=\ZZ e_1\oplus\ZZ f_1$ are two hyperbolic planes 
and $L$ is an even positive definite lattice. 
We choose $(e,e_1,...,f_1,f)$ as a basis of $M$, 
here $...$ denotes a basis of $L(-1)$.  
At the standard 1-dimensional cusp determined by the isotropic plane $\latt{e,e_1}$, $\cD(M)$ can be realized as the following tube domain 
$$
\cH(L)=\{Z=(\tau,\mathfrak{z},\omega)\in \HH\times (L\otimes\CC)\times \HH: 
(\im Z,\im Z)>0\}, 
$$
where $(\im Z,\im Z)=2\im \tau \im \omega - 
(\im \mathfrak{z},\im \mathfrak{z})_L$. 
In this setting, a Jacobi form  is a modular form for the Jacobi group $\Gamma^J(L)$ which is the parabolic subgroup of $\Orth^+ (M)$ preserving the isotropic plane $\latt{e,e_1}$ and acting trivially on $L$. This group is the semidirect product of $\SL_2(\ZZ)$ with the integral Heisenberg group of $L$.

\begin{definition}
For $k\in\ZZ$, $t\in \NN$, a holomorphic 
function $\varphi : \HH \times (L \otimes \CC) \rightarrow \CC$ is 
called a {\it weakly holomorphic} Jacobi form of weight $k$ and index $t$ associated to $L$,
if it satisfies 
\begin{align*}
\varphi \left( \frac{a\tau +b}{c\tau + d},\frac{\mathfrak{z}}{c\tau + d} 
\right)& = (c\tau + d)^k 
\exp{\left(i \pi t \frac{c(\mathfrak{z},\mathfrak{z})}{c 
\tau + d}\right)} \varphi ( \tau, \mathfrak{z} ), \quad \left(\begin{array}{cc} 
a & b \\ 
c & d
\end{array}
\right)   \in \SL_2(\ZZ), \\
\varphi (\tau, \mathfrak{z}+ x \tau + y)&= 
\exp{\bigl(-i \pi t ( (x,x)\tau +2(x,\mathfrak{z}))\bigr)} 
\varphi (\tau, \mathfrak{z} ), \quad x,y \in L,
\end{align*}
and if it has a Fourier expansion  
\begin{equation*}
\varphi ( \tau, \mathfrak{z} )= \sum_{n\geq n_0 }\sum_{\ell\in L^\vee}f(n,
\ell)q^n\zeta^\ell,
\end{equation*}
where $n_0\in \ZZ$, $q=e^{2\pi i \tau}$ and $\zeta^\ell=e^{2\pi i (\ell,
\mathfrak{z})}$. 
If the Fourier expansion of $\varphi$ satisfies the condition
$(f(n,\ell) \neq 0 \Longrightarrow n \geq 0 )$
then $\varphi$ is called a {\it weak} Jacobi form.  If $( f(n,\ell) \neq 0 
\Longrightarrow 2n - (\ell,\ell) \geq 0 )$ (resp. $>0$)
then $\varphi$ is called a {\it holomorphic}
(resp. {\it cusp}) Jacobi form. 
\end{definition}

We denote by $J^{!}_{k,L,t}$ (resp. $J^{w}_{k,L,t}$, $J_{k,L,t}$, 
$J_{k,L,t}^{\text{cusp}}$) the vector space of weakly holomorphic Jacobi 
forms (resp. weak, holomorphic or cusp Jacobi forms) 
of weight $k$ and index $t$ for $L$. 
The Jacobi forms due to Eichler--Zagier $J_{k,N}$
are identical to the Jacobi forms $J_{k,A_1, N}$ for the lattice  $A_1=\latt{\ZZ, 2x^2}$.
\smallskip

We next introduce the additive Jacobi lifting and Borcherds product. 
 
Let $\varphi \in J_{k,L,t}^{!}$.  For any positive integer $m$, one has
\begin{equation}
\varphi \lvert_{k,t}T_{-}(m)(\tau,
\mathfrak{z})=m^{-1}\sum_{\substack{ad=m,a>0\\ 0\leq b <d}}a^k \varphi
\left(\frac{a\tau+b}{d},a\mathfrak{z}\right) \in J_{k,L,mt}^{!},
\end{equation}
and the Fourier coefficients of 
$\varphi \lvert_{k,t}T_{-}(m)(\tau,\mathfrak{z})$ can be calculated by the formula
$$
f_m(n,\ell)=\sum_{\substack{a\in \NN\\ a \mid (n,\ell,m)}}a^{k-1} f
\left( \frac{nm}{a^2},\frac{\ell}{a}\right),
$$
where $a\mid(n,\ell,m)$ means that $a\mid (n,m)$ 
and $a^{-1}\ell\in L^\vee$. 

\begin{theorem}[see Theorem 3.1 in \cite{Gri94}]
Let $\varphi \in J_{k,L,1}$. Then the function 
$$ \Grit(\varphi)(Z)=f(0,0)G_k(\tau)+\sum_{m\geq 1}\varphi \lvert_{k,1}
T_{-}(m)(\tau,\mathfrak{z}) e^{2\pi i m\omega}
$$
is a modular form of weight $k$ for $\widetilde{\Orth}^+(2U\oplus L(-1))$. Moreover, this modular form is symmetric i.e.
$\Grit(\varphi)(\tau,\mathfrak{z}, \omega)=
\Grit(\varphi)(\omega,\mathfrak{z}, \tau)$.
\end{theorem}

The following is a variant of Borcherds product in terms of Jacobi forms. The main advantage of this version is that it is rather easy to compute the Fourier expansion at the standard 1-dimensional cusp and the first Fourier-Jacobi coefficient is given by a theta block.

\begin{theorem}[see Theorem 4.2 in \cite{Gri18} for details]
\label{th:Borcherds}
We fix an ordering $\ell >0$ in $L^\vee$ in a way similar to positive root systems (see the bottom of page 825 in \cite{Gri18}). Let 
$$
\varphi(\tau,\mathfrak{z})=\sum_{n\in\ZZ, \ell\in L^\vee}f(n,\ell)q^n 
\zeta^\ell \in J^{!}_{0,L,1}.
$$
Assume that $f(n,\ell)\in \ZZ$ for all $2n-(\ell,\ell)\leq 0$. 
There is a meromorphic modular form of weight 
$f(0,0)/2$ and character $\chi$ with respect to  
$\widetilde{\Orth}^+(2U\oplus L(-1))$ defined as
$$ 
\Borch(\varphi)=
\biggl(\Theta_{f(0,*)}
(\tau,\mathfrak{z})\exp{(2\pi i\, C\omega)}\biggr)
\exp \left(-\Grit(\varphi)\right),
$$
where
$C=\frac{1}{2\rank(L)}\sum_{\ell\in L^\vee}f(0,\ell)(\ell,\ell)$ and
\begin{equation}\label{FJtheta}
\Theta_{f(0,*)}(\tau,\mathfrak{z})
=\eta(\tau)^{f(0,0)}\prod_{\ell >0}
\biggl(\frac{\vartheta(\tau,(\ell,\mathfrak{z}))}{\eta(\tau)} 
\biggr)^{f(0,\ell)}
\end{equation}
is a general theta block.
The character $\chi$ is induced by the character of the theta-block
and by the relation $\chi(V)=(-1)^D$, where
$V: (\tau,\mathfrak{z}, \omega) \to (\omega,\mathfrak{z},\tau)$, 
and $D=\sum_{n<0}\sigma_0(-n) f(n,0)$.

The poles and zeros of $\Borch(\varphi)$ lie on the rational quadratic 
divisors $\cD_v$, where $v\in 2U\oplus L^\vee(-1)$ is a primitive vector 
with $(v,v)<0$. The multiplicity of this divisor is given by 
$$ \mult \cD_v = \sum_{d\in \ZZ,d>0 } f(d^2n,d\ell),$$
where $n\in\ZZ$, $\ell\in L^\vee$ such that 
$(v,v)=2n-(\ell,\ell)$ and $v-(0,0,\ell,0,0)\in 2U\oplus L(-1)$.

The same function has the following infinite product expansion
$$
\Borch(\varphi)(Z)=q^A \zeta^{\vec{B}} \xi^C\prod_{\substack{n,m\in\ZZ, \ell
\in L^\vee\\ (n,\ell,m)>0}}(1-q^n \zeta^\ell \xi^m)^{f(nm,\ell)}, 
$$ 
where $Z= (\tau,\mathfrak{z}, \omega) \in \cH (L)$, 
$q=\exp(2\pi i \tau)$, 
$\zeta^\ell=\exp(2\pi i (\ell, \mathfrak{z}))$, $\xi=\exp(2\pi i \omega)$,
the notation $(n,\ell,m)>0$ means that either $m>0$, or $m=0$ 
and $n>0$, or $m=n=0$ and $\ell<0$, and
$$
A=\frac{1}{24}\sum_{\ell\in L^\vee}f(0,\ell),\quad
\vec{B}=\frac{1}{2}\sum_{\ell>0} f(0,\ell)\ell.
$$
\end{theorem}

\begin{remark}\label{rem:divisor}
By the Eichler criterion (see \cite[Proposition 3.3]{GHS09}), if $v_1, v_2\in 2U\oplus L^\vee(-1)$ are 
primitive, have the same norm, and have the same image in the discriminant 
group, i.e. $v_1-v_2\in 2U\oplus L(-1)$, then there exists $g\in 
\widetilde{\Orth}^+(2U\oplus L(-1))$ such that $g(v_1)=v_2$.
Therefore, for a primitive vector $v\in 2U\oplus L^\vee(-1)$ with $
(v,v)<0$, there exists a vector $(0,n,\ell,1,0)\in  2U
\oplus L^\vee(-1)$ such that $(v,v)=2n-(\ell,\ell)$,
$v- (0,n,\ell,1,0)\in 2U\oplus L(-1)$ and  
$$
\widetilde{\Orth}^+(2U\oplus L(-1))\cdot 
\cD_v=\widetilde{\Orth}^+(2U\oplus L(-1))\cdot\cD_{(0,n,\ell,1,0)}.
$$
\end{remark}

\begin{remark}
By \cite[Lemma 2.1]{Gri94}, the Fourier coefficient $f(n,\ell)$ of $\varphi\in J_{k, L, 1}^{!}$ depends only on the (hyperbolic) norm $2n-(\ell,\ell)$ of its index
and the image of $\ell$ in the discriminant group of $L$. In other word, $f(n_1, \ell_1)=f(n_2, \ell_2)$ if  $2n_1-(\ell_1,\ell_1)=2n_2-(\ell_2,\ell_2)$ and if $\ell_1-\ell_2\in L$.

The divisor of the Borcherds product in Theorem \ref{th:Borcherds}
is determined by the so-called  {\it singular} Fourier coefficients 
$f(n,\ell)$ with $2n-(\ell,\ell)<0$.
There are only a finite number of orbits of such coefficients that are 
supported because the norm  $2n-(\ell,\ell)$ of  the indices of the nontrivial Fourier coefficients is bounded from below.  
\end{remark}

\begin{remark}
Paramodular forms of weight $k$ for $\Gamma_N$ can be regarded as modular forms of the same weight for $\widetilde{\Orth}^+(2U\oplus \latt{-2N})$ (see \cite{GN98}). Thus we can use the pullback of orthogonal modular forms to construct paramodular forms.
\end{remark}

At the end of this section, we recall the isomorphism between vector-valued modular forms and Jacobi forms. Let $D$ be a discriminant form. Let $\{\textbf{e}_\gamma: \gamma \in D\}$ be the basis of the group ring $\CC[D]$. The \textit{Weil representation} of $\Mp_2(\ZZ)$ (i.e. the double covering of $\SL_2(\ZZ)$) on $\CC[D]$ is a unitary representation defined by the action of the generators of $\Mp_2(\ZZ)$ as follows (see \cite{Bru02})
\begin{align}
\rho_D(T)\textbf{e}_\gamma &= \exp(-\pi i \gamma^2)\textbf{e}_\gamma,\\
\rho_D(S)\textbf{e}_\gamma &= \frac{\exp(\pi i \sign(D)/4)}{\sqrt{\abs{D}}} \sum_{\beta\in D}\exp(2\pi i(\gamma,\beta))\textbf{e}_\beta.
\end{align}
Let $f(\tau)=\sum_{\gamma\in D}f_\gamma(\tau)\textbf{e}_\gamma$ be a holomorphic function on $\HH$ with values in $\CC[D]$ and $k\in \frac{1}{2}\ZZ$. The function $f$ is called a nearly holomorphic modular form for $\rho_D$ of weight $k$ if 
$$f(A\tau)=\phi(\tau)^{2k}\rho_D(A)f(\tau), \quad \forall (A,\phi) \in \Mp_2(\ZZ)$$
and if $f$ is meromorphic at $i\infty$. If $f$ is also holomorphic at $i\infty$, then it is called a holomorphic modular form. If $f$ vanishes at $i\infty$, then it is called a cusp form. The modular form $f$ has a Fourier expansion of the form
\begin{equation}
f(\tau)=\sum_{\gamma\in D}\sum_{n\in \ZZ-\gamma^2/2}c_\gamma(n)e^{2\pi i n\tau}\textbf{e}_\gamma.
\end{equation}
Here, the sum 
$\sum_{\gamma\in D}\sum_{n<0}c_\gamma(n)e^{2\pi i n\tau}\textbf{e}_\gamma$
is called the principal part of $f$.

The orthogonal group $\Orth(D)$ acts on $\CC[D]$ via 
$
\sigma\left(\sum_{\gamma\in D}a_\gamma \textbf{e}_\gamma  \right)=\sum_{\gamma\in D}a_\gamma \textbf{e}_{\sigma(\gamma)}
$
and this action commutes with that of $\rho_D$ on $\CC[D]$. Thus $\Orth(D)$ acts well on modular forms for the Weil representation. 

Let $L$ be an even positive definite lattice with discriminant form $D(L)$. The theta series associated to $L$ is defined as
$$
\Theta_{\gamma}^{L}(\tau,\mathfrak{z})=\sum_{\ell \in \gamma +L}\exp\left(\pi i(\ell,\ell) \tau + 2\pi i(\ell,\mathfrak{z}) \right), \quad \gamma\in D(L).
$$
The map 
\begin{equation}
F(\tau)=\sum_{\gamma\in D(L)}F_\gamma(\tau)\textbf{e}_\gamma \longmapsto \sum_{\gamma\in D(L)}F_\gamma(\tau)\Theta_\gamma^L(\tau,\mathfrak{z})
\end{equation}
defines an isomorphism between the space of nearly holomorphic modular forms of weight $k$ for $\rho_{D(L)}$ and the space of weakly holomorphic Jacobi forms of weight $k+\rank(L)/2$ and index $1$ for $L$. The principal part of $F$ corresponds to the singular Fourier coefficients of the Jacobi form. Hence the map also induces an isomorphism between the subspaces of holomorphic modular (resp. cusp) forms for $\rho_{D(L)}$ and holomorphic (resp. cusp) Jacobi forms of index $1$ for $L$.

\section{Proof of Theorem \ref{th:wt2}}
The aim of this section is to prove Theorem \ref{th:wt2}. Following the strategy mentioned in the introduction, by seeking  Scheithauer's list of reflective modular forms in \cite{Sch06}, we found that there is a strongly reflective modular form of singular weight $2$ with the complete 4-reflective divisors, 10-reflective divisors and 20-reflective divisors on the lattice 
$M_{2,6}=U\oplus U(10)\oplus A_4(-1)$. We denote this modular form by $\Psi_2^{\Sch}$. Note that $\Psi_2^{\Sch}$ was constructed by Scheithauer as the Borcherds product of a vector-valued modular form which is a lifting of the eta-quotient $\eta^{-1}(\tau)\eta^{-2}(2\tau)\eta^{-3}(5\tau)\eta^{2}(10\tau)$. The maximal modular group of $\Psi_2^{\Sch}$ is the full modular group $\Orth^+(M_{2,6})$. The divisor of $\Psi_2^{\Sch}$ is as follows
\begin{equation}\label{eq:div2}
\div(\Psi_2^{\Sch})=\sum_{\substack{v\in M_{2,6}^\vee \; \text{primitive}\\ (v,v)=-1,\, \ord (v)= 2}} \cD_v +  \sum_{\substack{v\in M_{2,6}^\vee \; \text{primitive}\\ (v,v)=-\frac{2}{5},\, \ord (v)= 5}} \cD_v + \sum_{\substack{v\in M_{2,6}^\vee \; \text{primitive}\\ (v,v)=-\frac{1}{5},\, \ord (v)= 10}} \cD_v.
\end{equation}  
The lattice $M_{2,6}(-1)$ is of level $10$ and has genus $\II_{6,2}(2_{\II}^{+2}5^{+3})$. It is clear that the discriminant group of $M_{2,6}$ has 3 generators. By \cite{Nik80} (or \cite[Lemma 2.3]{Wan19}), there exists an even positive definite lattice $L$ of rank 4 such that $M_{2,6}\cong 2U\oplus L(-1)$. By \cite{LMFDB}, the genus $\II_{4,0}(2_{\II}^{+2}5^{+3})$ contains only one class and the label of this lattice is $4.500.10.1.2$. Thus the lattice $L$ is unique up to isomorphism and we denote it by $L_4$.  The  matrix model of  $L_4$ and its inverse are respectively 
\begin{align*}
&L_4=\left(\begin{array}{cccc}
4 & 2 & 2 & 2 \\ 
2 & 6 & 1 & 1 \\ 
2 & 1 & 6 & 1 \\ 
2 & 1 & 1 & 6
\end{array}  \right),& &L_4^{-1}=\frac{1}{10}\left(\begin{array}{cccc}
4 & -1 & -1 & -1 \\ 
-1 & 2 & 0 & 0 \\ 
-1 & 0 & 2 & 0 \\ 
-1 & 0 & 0 & 2
\end{array}  \right).&
\end{align*}
Let $\alpha_1, \alpha_2, \alpha_3, \alpha_4$ be a basis of $L_4$ corresponding to the above matrix and $w_1, w_2, w_3, w_4$ be the associated dual basis. The lattice $L_4$ is of level $10$ and has determinant $500$. We next consider $\Psi_2^{\Sch}$ as a reflective modular form on $2U\oplus L_4(-1)$. Then $\Psi_2^{\Sch}$ should be a Borcherds product of a Jacobi form $\Psi_{L_4} \in J_{0,L_4,1}^!$. From the divisor of $\Psi_2^{\Sch}$, we conclude that $\Psi_{L_4}$ is in fact a weak Jacobi form because $\Psi_2^{\Sch}$ has no 2-reflective divisor i.e. $\cD_r$ with $r\in M_{2,6}$ and $(r,r)=-2$. Then $\Psi_{L_4}$ has a Fourier expansion of the form
$$
\Psi_{L_4}(\tau,\mathfrak{z})=\sum_{n\in \NN, \ell\in L_{4}^\vee}f(n,\ell)q^n\zeta^\ell.
$$
By Theorem \ref{th:Borcherds}, $f(0,0)=4$, and for any $\ell\neq 0$, the $q^0$-term $f(0,\ell)\zeta^\ell$ determines a divisor $\cD_{(0,0,\ell,1,0)}$. But $\cD_{(0,0,\ell,1,0)}$ must be reflective. Therefore either $f(0,\ell)=0$, or $f(0,\ell)=1$ and $\ell$ satisfies one of the following conditions
\begin{itemize}
\item[(a)] $(\ell,\ell)=1$ and $\ell$ has order $2$ in $L_4^\vee/L_4$; 
\item[(b)] $(\ell,\ell)=\frac{2}{5}$ and $\ell$ has order $5$ in $L_4^\vee/L_4$; 
\item[(c)] $(\ell,\ell)=\frac{1}{5}$ and $\ell$ has order $10$ in $L_4^\vee/L_4$.
\end{itemize}
By direct calculations, up to sign
\begin{enumerate}
\item the vectors of type (a) are $2w_1+w_2+w_3+w_4$;
\item the vectors of type (b) are $w_2+w_3$, $w_2-w_3$, $w_2+w_4$, $w_2-w_4$, $w_3+w_4$, $w_3-w_4$;
\item the vectors of type (c) are $w_2$, $w_3$, $w_4$.
\end{enumerate}
These vectors determine completely the $q^0$-term of $\Psi_{L_4}$. By Theorem \ref{th:Borcherds}, the first Fourier-Jacobi coefficient of $\Psi_2^{\Sch}$ is known. In the coordinates $\mathfrak{z}=z_1\alpha_1+z_2\alpha_2+z_3\alpha_3+z_4\alpha_4$, this leading Fourier-Jacobi coefficient can be written as
\begin{align*}
\Theta_{L_4}(\tau,\mathfrak{z})=\eta^{-6}&\vartheta(2z_1+z_2+z_3+z_4)\vartheta(z_2-z_3)\vartheta(z_2+z_3)\vartheta(z_2-z_4)\\
&\vartheta(z_2+z_4)\vartheta(z_3-z_4)\vartheta(z_3+z_4)\vartheta(z_2)\vartheta(z_3)\vartheta(z_4),
\end{align*}
where $\vartheta(z)=\vartheta(\tau,z)$.
This function is a pure theta block of weight 2 with vanishing order one in $q$ and it defines a holomorphic Jacobi form of weight 2 and index 1 for $L_4$.

\begin{theorem}\label{th:wt2-4}
The Borcherds product $\Psi_2^{\Sch}=\Borch(\Psi_{L_4})$ is a Gritsenko lift. In other word, 
$$
\Psi_2^{\Sch}=\Borch(\Psi_{L_4})=\Grit(\Theta_{L_4}) \quad \text{and} \quad \Psi_{L_4}=-\frac{\Theta_{L_4}\lvert T_{-}(2)}{\Theta_{L_4}}.
$$
\end{theorem}

\begin{proof}
Firstly, if $\Borch(\Psi_{L_4})=\Grit(\Theta_{L_4})$ then we have $\Psi_{L_4}=-(\Theta_{L_4}\lvert T_{-}(2))/\Theta_{L_4}$  by the expressions of  Gritsenko lifts and Borcherds products. Thus we only need to show $\Psi_2^{\Sch}=\Grit(\Theta_{L_4})$. It suffices to prove $\div(\Grit(\Theta_{L_4}))\supset\div( \Psi_2^{\Sch})$ by K\"ocher's principle. According to \cite[Lemma 3.8]{GW18}, $\Grit(\Theta_{L_4})$ vanishes on the reflective divisors $\cD_{(0,0,\ell,1,0)}$ for all vectors $\ell$ of type (a), (b) and (c). But there are some other reflective divisors in $\div( \Psi_2^{\Sch})$. In fact, in the discriminant group of $L_4$, there are one class of norm $1$ ($\m 2$) and order 2, twenty classes of norm $\frac{2}{5}$ ($\m 2$) and order 5, and thirty-one classes of norm $\frac{1}{5}$ ($\m 2$) and order 10 (see \cite[\S 3]{Sch06}). Thus $L_4$ does not satisfy the condition $\norm_2$ and the argument in \cite{GW18} does not work in this case. Fortunately, the level of $L_4$ is squarefree and  $\Grit(\Theta_{L_4})$ vanishes on one part of reflective divisors. If we can prove that the maximal modular group of $\Grit(\Theta_{L_4})$ is $\Orth^+(M_{2,6})$, then we conclude from Lemma \ref{lem:orbit} below that $\Grit(\Theta_{L_4})$ vanishes on all reflective divisors, which implies $\div(\Grit(\Theta_{L_4}))\supset\div( \Psi_2^{\Sch})$. This assertion is proved in Lemma \ref{lem:full}.
\end{proof}

\begin{lemma}\label{lem:orbit}
Suppose that $M=2U\oplus L(-1)$ is of squarefree level. If $u$, $v$ are primitive vectors in $M^\vee$ satisfying $(u,u)=(v,v)$ and $\ord(u)=\ord(v)$, then there exists $g\in \Orth^+(M)$ such that $g(u)=v$.
\end{lemma}

\begin{proof}
Firstly, we view $u$ and $v$ as classes in the discriminant group $D(M)$. Since the level of $M$ is squarefree, by Proposition 5.1 and the paragraph after Proposition 5.2 in \cite{Sch15}, there exists $\hat{\sigma}\in \Orth(D(M))$ such that $\hat{\sigma}(u)=v$ in $D(M)$. According to \cite[Theorem 1.14.2]{Nik80}, the natural homomorphism $\mathfrak{J}: \Orth(M)\to \Orth(D(M))$ is surjective. It is easy to check $\mathfrak{J}(\Orth(M))=\mathfrak{J}(\Orth^+(M))$ in our case. Thus there is an automorphism $\sigma\in \Orth^+(M)$ such that its image under $\mathfrak{J}$ is $\hat{\sigma}$. Then we have $\sigma(u)-v\in M$. By the Eichler criterion (see Remark \ref{rem:divisor}), there exists $h\in \widetilde{\Orth}^+(M)$ such that $h(\sigma(u))=v$. Hence the automorphism $h\circ \sigma$ is the desired $g$.
\end{proof}

\begin{lemma}
The space $J_{2,L_4,1}$ has dimension $2$ and it has the following basis
\begin{align*}
\Theta_{L_4}^{(1)}(\tau,\mathfrak{z})=\eta^{-6}&\vartheta(z_1+z_2+z_3+z_4)\vartheta(z_2-z_3)\vartheta(z_2+z_3)\vartheta(z_2-z_4)\vartheta(z_2+z_4)\\
&\vartheta(z_3-z_4)\vartheta(z_3+z_4)\vartheta(z_1+z_2)\vartheta(z_1+z_3)\vartheta(z_1+z_4),\\
\Theta_{L_4}^{(2)}(\tau,\mathfrak{z})=\eta^{-6}&\vartheta(z_1)\vartheta(z_2-z_3)\vartheta(z_2+z_3)\vartheta(z_2-z_4)\vartheta(z_2+z_4)\vartheta(z_3-z_4)\\
&\vartheta(z_3+z_4)\vartheta(z_1+z_2+z_3)\vartheta(z_1+z_2+z_4)\vartheta(z_1+z_3+z_4).
\end{align*}
Correspondingly, the space $M_0(\rho_{D(L_4)})$ of modular forms of weight $0$ for $\rho_{D(L_4)}$ has dimension $2$. The modular forms corresponding to the two Jacobi forms form a basis of this space.  
\end{lemma}

\begin{proof}
We can write $M_{2,6}(-1)=2U\oplus L_4 =U\oplus U(2)\oplus A_4^\vee(5)$. By \cite[Remark 3.2]{ES17}, we have $\dim M_0(\rho_{D(L_4)})=\dim M_0(\rho_{D(U(2))})\times\dim M_0(\rho_{D(A_{4}^\vee(5))})$, where the latter two spaces correspond to the $2$-part and $5$-part defined in \cite{ES17}. Notice that our Weil representation for $D(L)$ is equal to the Weil representation for $D(L(-1))$ used in \cite{ES17}. By \cite[Table 1 and Table 5]{ES17}, we have $\dim M_0(\rho_{D(U(2))})=2$ and $\dim M_0(\rho_{D(A_{4}^\vee(5))})=1$. Thus $\dim M_0(\rho_{D(L_4)})=2$. We next construct the basis using the pullback of Jacobi forms. By \cite[(3.11)]{GW18}, there is a Jacobi form $\Theta_{A_4}\in J_{2,A_4^\vee(5),1}$. There are two embeddings from $L_4$ into $A_4^\vee(5)$.  Firstly, the four vectors $(\alpha_2+\alpha_3+\alpha_4-\alpha_1)/2$, $\alpha_2$, $(\alpha_2+\alpha_3-\alpha_4+\alpha_1)/2$, $\alpha_1$ form a basis of $A_4^\vee(5)$ and the corresponding pullback of $\Theta_{A_4}$ gives $\Theta_{L_4}^{(1)}$. Secondly, the four vectors $(\alpha_2-\alpha_3+\alpha_4)/2$, $\alpha_2$, $(\alpha_2+\alpha_3+\alpha_4)/2$, $\alpha_1$  form another basis of $A_4^\vee(5)$ and the corresponding pullback of $\Theta_{A_4}$ gives $\Theta_{L_4}^{(2)}$.
\end{proof}

\begin{remark}\label{rem:anotherconstruction}
We can also construct the above basis using vector-valued modular forms. Let us write $M_{2,6}(-1)=U\oplus U_1(2)\oplus A_4^\vee(5)< U\oplus U_1 \oplus A_4^\vee(5)$. The unique modular form of weight $0$ for $\rho_{D(A_4^\vee(5))}$ corresponds to the Jacobi form $\Theta_{A_4}$ and we denote this modular form by  
$$
F(\tau)=\sum_{\gamma\in D(A_4^\vee(5))}a_\gamma\textbf{e}_\gamma,
$$
where $a_\gamma \in \ZZ$. Let $e$, $f$ be a basis of $U_1(2)$, i.e. $(e,e)=(f,f)=0$ and $(e,f)=2$. Then the corresponding dual basis of $U_1(2)$ is $\{\frac{1}{2} f, \frac{1}{2}e\}$.  We can choose the basis of $U_1$ as $\{ \frac{1}{2} e,f\}$ or $\{ e,  \frac{1}{2}f\}$.   We use the lifting $\big\uparrow_{U_1\oplus A_4^\vee(5)}^{U_1(2)\oplus A_4^\vee(5)}$ of  \cite[Corollary 2.2]{Ma18b} (or see \cite[Theorem 5.3]{Bor98},  \cite[Lemmas 5.6, 5.7]{Bru02}) to construct modular forms for $\rho_{D(L_4)}=\rho_{D(U_1(2)\oplus A_4^\vee(5))}$. Under the notations of \cite{Ma18b},  we can construct two such modular forms of weight $0$ because there are two choices of the basis of $U_1\oplus A_4^\vee(5)$. The two modular forms are constructed as
\begin{align*}
F^{(1)}(\tau)&=\sum_{\gamma\in D(A_4^\vee(5))}a_\gamma (\textbf{e}_\gamma+ \textbf{e}_{\frac{1}{2} e+\gamma}), \quad \text{where the basis of $U_1$ is chosen as $\{\frac{1}{2} e,f\}$},\\
F^{(2)}(\tau)&=\sum_{\gamma\in D(A_4^\vee(5))}a_\gamma (\textbf{e}_\gamma+ \textbf{e}_{\frac{1}{2}f+\gamma}), \quad \text{where the basis of $U_1$ is chosen as $\{e, \frac{1}{2}f\}$}.
\end{align*}

\end{remark}

\begin{lemma}\label{lem:fullJ}
The vector-valued modular form corresponding to $\Theta_{L_4}$ is invariant under the orthogonal group $\Orth(D(L_4))$ up to a character of order $2$ and we have $\Theta_{L_4}=\Theta_{L_4}^{(1)}-\Theta_{L_4}^{(2)}$.
\end{lemma}

\begin{proof}
We see from the above remark that there are both nonzero Fourier coefficients of order $5$ and those of order $10$ in $F^{(1)}$ and $F^{(2)}$. But their difference $F^{(1)}-F^{(2)}$ has only nonzero Fourier coefficients of order $10$. Since the dimension of the corresponding space is $2$, we conclude that $F^{(1)}-F^{(2)}$ is invariant under $\Orth(D(L_4))$ up to a character. Moreover, the character has order $2$ because the modular form has integral Fourier coefficients. We prove the identity in the lemma by comparing their first Fourier coefficients.
\end{proof}

\begin{remark}
The identity $\Theta_{L_4}=\Theta_{L_4}^{(1)}-\Theta_{L_4}^{(2)}$ is in fact a direct consequence of the following variant of the Riemann theta relation (see \cite[Page 20, ($R_5$)]{Mum83})
\begin{equation}\label{eq:RTR}
\prod_{j=1}^4 \vartheta(\tau,z_j)+\prod_{j=1}^4 \vartheta(\tau,m_j)=\prod_{j=1}^4 \vartheta(\tau,p_j),
\end{equation}
where $(z_1,z_2,z_3,z_4)\in\CC^4$,  $(m_1,m_2,m_3,m_4)=(z_1,z_2,z_3,-z_4)A$, $(p_1,p_2,p_3,p_4)=(z_1,z_2,z_3,z_4)A$, and
$$
A=\frac{1}{2} \left( \begin{array}{cccc}
1 & 1 & 1 & 1 \\ 
1 & 1 & -1 & -1 \\ 
1 & -1 & 1 & -1 \\ 
1 & -1 & -1 & 1
\end{array}  
\right).
$$
A special case of the above identity can be found in \cite[Proposition 4.3]{BPY16}. We remark that the identity \eqref{eq:RTR} can also be viewed as a relation between Jacobi forms for the root lattice $D_4$ (see \cite[Example 2.8]{CG13}).
\end{remark}

\begin{lemma}\label{lem:full}
The function $\Grit(\Theta_{L_4})$ is a modular form of weight $2$ for $\Orth^+(M_{2,6})$ with a character of order $2$.
\end{lemma}

\begin{proof}
In \cite[Theorem 14.3]{Bor98}, Borcherds reconstructed the Gritsenko lift in the context of modular forms for the Weil representation. In the Borcherds theorem, the Gritsenko lift is constructed as the integral of the inner product of a vector-valued modular form (i.e. the input) with the Siegel theta function over the fundamental domain. Notice that the Siegel theta function is invariant under the orthogonal group of the lattice. Thus, if the input is invariant under the orthogonal group of the discriminant form, then the corresponding Gritsenko lift is a modular form for the full modular group. We then finish the proof by Lemma \ref{lem:fullJ}.
\end{proof}

\begin{corollary}\label{cor:wt2-4}
We have the following equality
$$
\Borch\left(-\frac{\Theta_{L_4}|T_{-}(2)}{\Theta_{L_4}}  \right)= \Borch\left(-\frac{\Theta_{L_4}^{(1)}|T_{-}(2)}{\Theta_{L_4}^{(1)}}  \right) - \Borch\left(-\frac{\Theta_{L_4}^{(2)}|T_{-}(2)}{\Theta_{L_4}^{(2)}}  \right).
$$
\end{corollary}

\begin{proof}
The main theorem of \cite{GW18} says that $\Grit(\Theta_{A_4})$ is a Borcherds product. From the constructions, we observe that $\Theta_{L_4}^{(1)}$ and $\Theta_{L_4}^{(2)}$ are variants of $\Theta_{A_4}$. Thus, $\Grit(\Theta_{L_4}^{(1)})$ and $\Grit(\Theta_{L_4}^{(2)})$ are all Borcherds products. Since the Gritsenko lift is additive, we conclude from Lemma \ref{lem:fullJ} that $\Grit(\Theta_{L_4})=\Grit(\Theta_{L_4}^{(1)})-\Grit(\Theta_{L_4}^{(2)})$. By Theorem \ref{th:wt2-4}, $\Grit(\Theta_{L_4})$ is also a Borcherds product. Then we get the expected identity for Borcherds products.
\end{proof}

We now prove the first main theorem.
\begin{proof}[Proof of Theorem \ref{th:wt2}]
Similar to the proof of Theorem 1.1 in \cite[\S 4.1]{GW18}, we use the specializations of the identity in Theorem \ref{th:wt2-4} to prove this result. 
By taking $\mathfrak{z}=z(a_1-a_3-a_4)\alpha_1+ z(a_2+a_3+a_4)\alpha_2+z(a_3+a_4)\alpha_3+za_4\alpha_4$, we finish the proof.
\end{proof}

As a direct consequence of Corollary \ref{cor:wt2-4}, we have the following.
\begin{corollary}
Let $\mathbf{a}=(a_1, a_2, a_3, a_4)\in \ZZ^4$. We define two theta blocks
\begin{align*}
\varphi^{(1)}_{2,\mathbf{a}}&=\eta^{-6}
\vartheta_{a_1}\vartheta_{a_2}\vartheta_{a_2+a_3}\vartheta_{a_2+2a_3+2a_4}
\vartheta_{a_1+a_2}\vartheta_{a_2+a_3+2a_4}\vartheta_{a_3}
\vartheta_{a_1-a_3}\vartheta_{a_3+2a_4}\vartheta_{a_1+a_2+a_3+2a_4},\\
\varphi^{(2)}_{2,\mathbf{a}}&=\eta^{-6}
\vartheta_{a_1-a_3-a_4}\vartheta_{a_2}\vartheta_{a_2+a_3}\vartheta_{a_2+2a_3+2a_4}
\vartheta_{a_1+a_2+a_3+a_4}\vartheta_{a_2+a_3+2a_4}\vartheta_{a_3}
\vartheta_{a_1+a_4}\vartheta_{a_3+2a_4}\vartheta_{a_1+a_2+a_4}.
\end{align*}
For any $\mathbf{a}\in\ZZ^4$ such that none of $\varphi_{2,\mathbf{a}}$, $\varphi^{(1)}_{2,\mathbf{a}}$ and $\varphi^{(2)}_{2,\mathbf{a}}$ is identically zero, we have 
$$
\varphi_{2,\mathbf{a}}=\varphi^{(1)}_{2,\mathbf{a}}-\varphi^{(2)}_{2,\mathbf{a}},
$$
$$
\Borch\left(-\frac{\varphi_{2,\mathbf{a}}|T_{-}(2)}{\varphi_{2,\mathbf{a}}}
\right)=\Borch\left(-\frac{\varphi^{(1)}_{2,\mathbf{a}}|T_{-}(2)}{\varphi^{(1)}_{2,\mathbf{a}}}
\right)-\Borch\left(-\frac{\varphi^{(2)}_{2,\mathbf{a}}|T_{-}(2)}{\varphi^{(2)}_{2,\mathbf{a}}}
\right).
$$
\end{corollary}

The first example of the corollary is given by $\mathbf{a}=(3,1,1,1)$. In this case,  the index of Jacobi forms is $67$. The corresponding theta blocks are
$$
\eta^{-6}\vartheta^3\vartheta_2^2\vartheta_3^2\vartheta_4\vartheta_5\vartheta_8=\eta^{-6}\vartheta^2\vartheta_2^2\vartheta_3^2\vartheta_4^2\vartheta_5\vartheta_7-\eta^{-6}\vartheta^3\vartheta_2\vartheta_3\vartheta_4^2\vartheta_5^2\vartheta_6.
$$

\section{Proof of Theorem \ref{th:wt3}}
This section aims to prove Theorem \ref{th:wt3}. The proof is similar to the proof of Theorem \ref{th:wt2} in the previous section. 

Analysing Scheithauer's list of reflective modular forms (see \cite{Sch06}), we found that there is a strongly reflective modular form of singular weight $3$ with the complete 4-reflective divisors, 6-reflective divisors and 12-reflective divisors for the lattice 
$M_{2,8}=U\oplus U(6)\oplus E_6(-1)$. We denote this modular form by $\Psi_3^{\Sch}$. Note that $\Psi_3^{\Sch}$ was constructed by Scheithauer as the Borcherds product of a vector-valued modular form which is a lifting of the eta-quotient $\eta^{-1}(\tau)\eta^{-4}(2\tau)\eta^{-5}(3\tau)\eta^{4}(6\tau)$. The maximal modular group of $\Psi_3^{\Sch}$ is the full modular group $\Orth^+(M_{2,8})$. The divisor of $\Psi_3^{\Sch}$ is 
\begin{equation}\label{eq:div3}
\div(\Psi_3^{\Sch})=\sum_{\substack{v\in M_{2,8}^\vee \; \text{primitive}\\ (v,v)=-1,\, \ord (v)= 2}} \cD_v +  \sum_{\substack{v\in M_{2,8}^\vee \; \text{primitive}\\ (v,v)=-\frac{2}{3},\, \ord (v)= 3}} \cD_v + \sum_{\substack{v\in M_{2,8}^\vee \; \text{primitive}\\ (v,v)=-\frac{1}{3},\, \ord (v)= 6}} \cD_v.
\end{equation} 
The lattice $M_{2,8}(-1)$ is of level $6$ and has genus $\II_{8,2}(2_{\II}^{+2}3^{-3})$. It is clear that the discriminant group of $M_{2,8}$ has 3 generators. By \cite{Nik80}, there exists an even positive definite lattice $L$ of rank $6$ such that $M_{2,8}\cong 2U\oplus L(-1)$. By \cite{LMFDB}, the genus $\II_{6,0}(2_{\II}^{+2}3^{-3})$ conatins only one class and the label of this lattice is $6.108.6.1.1$. Thus, the lattice $L$ is unique up to isomorphism and we denote it by $L_6$.  The  matrix model of  $L_6$ and its inverse are respectively 
\begin{align*}
&L_6=\left( \begin{array}{cccccc}
4 & 2 & 0 & 0 & -2 & 0 \\ 
2 & 4 & 0 & 0 & -1 & 0 \\ 
0 & 0 & 2 & -1 & 0 & 0 \\ 
0 & 0 & -1 & 2 & 0 & 0 \\ 
-2 & -1 & 0 & 0 & 2 & 1 \\ 
0 & 0 & 0 & 0 & 1 & 4
\end{array} \right),& &
L_6^{-1}=\frac{1}{6}\left( \begin{array}{cccccc}
4 & -1 & 0 & 0 & 4 & -1 \\ 
-1 & 2 & 0 & 0 & 0 & 0 \\ 
0 & 0 & 4 & 2 & 0 & 0 \\ 
0 & 0 & 2 & 4 & 0 & 0 \\ 
4 & 0 & 0 & 0 & 8 & -2 \\ 
-1 & 0 & 0 & 0 & -2 & 2
\end{array} \right).&
\end{align*}
Let $\beta_i$, $1\leq i \leq 6$, be a basis of $L_6$ corresponding to the above matrix and $u_i$, $1\leq i \leq 6$, be the associated dual basis.
The lattice $L_6$ has level $6$ and determinant $108$. We next consider $\Psi_3^{\Sch}$ as a reflective modular form on $2U\oplus L_6(-1)$. Then $\Psi_3^{\Sch}$ is a Borcherds product of a weak Jacobi form $\Psi_{L_6} \in J_{0,L_6,1}^w$. We then assume that $\Psi_{L_6}$ has a Fourier expansion of the form
$$
\Psi_{L_6}(\tau,\mathfrak{z})=\sum_{n\in \NN, \ell\in L_{6}^\vee}f(n,\ell)q^n\zeta^\ell.
$$
By Theorem \ref{th:Borcherds}, $f(0,0)=6$, and for any $\ell\neq 0$, the $q^0$-term $f(0,\ell)\zeta^\ell$ determines a divisor $\cD_{(0,0,\ell,1,0)}$ which must be reflective. Then we have that either $f(0,\ell)=0$, or $f(0,\ell)=1$ and $\ell$ satisfies one of the following conditions
\begin{itemize}
\item[(i)] $(\ell,\ell)=1$ and $\ell$ has order $2$ in $L_6^\vee/L_6$; 
\item[(ii)] $(\ell,\ell)=\frac{2}{3}$ and $\ell$ has order $3$ in $L_6^\vee/L_6$; 
\item[(iii)] $(\ell,\ell)=\frac{1}{3}$ and $\ell$ has order $6$ in $L_6^\vee/L_6$.
\end{itemize}
By direct calculations, up to sign
\begin{enumerate}
\item the vectors of type (i) are $2u_1+u_2-u_5$, $u_5+u_6$;
\item the vectors of type (ii) are $u_3$, $u_4$, $u_3-u_4$, $u_2-u_6$, $u_2+u_6$;
\item the vectors of type (iii) are $u_2$, $u_6$.
\end{enumerate}
We now have determined the $q^0$-term of $\Psi_{L_6}$. By Theorem \ref{th:Borcherds}, the first Fourier-Jacobi coefficient of $\Psi_3^{\Sch}$ is known to be a pure theta block. In the coordinates $\mathfrak{z}=\sum_{i=1}^6 z_i\beta_i$, this theta block can be written as
\begin{align*}
\Theta_{L_6}(\tau,\mathfrak{z})=\eta^{-3}&\vartheta(z_3)\vartheta(z_4)\vartheta(z_3-z_4)\vartheta(z_2-z_6)\vartheta(z_2+z_6)\\
&\vartheta(2z_1+z_2-z_5)\vartheta(z_5+z_6)\vartheta(z_2)\vartheta(z_6).
\end{align*}
This function defines a holomorphic Jacobi form of weight 3 and index 1 for $L_6$.

\begin{theorem}\label{th:wt3-6}
The Borcherds product $\Psi_3^{\Sch}=\Borch(\Psi_{L_6})$ is a Gritsenko lift. In other word, 
$$
\Psi_3^{\Sch}=\Borch(\Psi_{L_6})=\Grit(\Theta_{L_6}) \quad \text{and} \quad \Psi_{L_6}=-\frac{\Theta_{L_6}\lvert T_{-}(2)}{\Theta_{L_6}}.
$$
\end{theorem}

\begin{proof}
Similar to the proof of Theorem \ref{th:wt2-4}, we only need to show that the maximal modular group of $\Grit(\Theta_{L_6})$ is the full modular group. We will prove this in Lemma \ref{lem:full3} below.
\end{proof}

\begin{lemma}
The space $J_{2,L_6,1}$ has dimension $2$ and it has the following basis
\begin{align*}
\Theta_{L_6}^{(1)}(\tau,\mathfrak{z})=\eta^{-3}&\vartheta(z_3)\vartheta(z_4)\vartheta(z_3-z_4)\vartheta(z_2-z_6)\vartheta(z_2+z_6)\\
&\vartheta(z_1)\vartheta(z_1+z_2+z_6)\vartheta(z_1-z_5)\vartheta(z_1+z_2-z_5-z_6),\\
\Theta_{L_6}^{(2)}(\tau,\mathfrak{z})=\eta^{-3}&\vartheta(z_3)\vartheta(z_4)\vartheta(z_3-z_4)\vartheta(z_2-z_6)\vartheta(z_2+z_6)\\
&\vartheta(z_1-z_5-z_6)\vartheta(z_1+z_2-z_5)\vartheta(z_1+z_2)\vartheta(z_1+z_6).
\end{align*}
Correspondingly, the space $M_0(\rho_{D(L_6)})$ of modular forms of weight $0$ for $\rho_{D(L_6)}$ has dimension $2$. The modular forms corresponding to the above two Jacobi forms form a basis of this space.  
\end{lemma}

\begin{proof}
We first write $M_{2,8}(-1)=2U\oplus L_6 =U\oplus U(2)\oplus 3A_2$. By \cite[Remark 3.2]{ES17}, we have $\dim M_0(\rho_{D(L_6)})=\dim M_0(\rho_{D(U(2))})\times\dim M_0(\rho_{D(3A_2)})$, where the latter two spaces correspond to the $2$-part and $3$-part in \cite{ES17}.  By \cite[Table 1 and Table 3]{ES17}, we know $\dim M_0(\rho_{D(U(2))})=2$ and $\dim M_0(\rho_{D(3A_2)})=1$. Thus $\dim M_0(\rho_{D(L_6)})=2$. We next construct the basis using the pullback of Jacobi forms. By \cite[\S 5.4]{Gri18} or \cite[Theorem 13.5]{GSZ19}, there is a Jacobi form $\Theta_{3A_2}\in J_{2,3A_2,1}$. There are two embeddings from $L_6$ into $3A_2$.  Firstly, the six vectors  $(\beta_2+\beta_6)/2$, $-(\beta_1+\beta_5)$, $\beta_3$, $\beta_4$, $\beta_5$ and $(\beta_2-\beta_6)/2$  form a basis of $3A_2$ and the corresponding pullback of $\Theta_{3A_2}$ is $\Theta_{L_6}^{(1)}$. Secondly, the six vectors  $(-\beta_1+\beta_2+\beta_6)/2$, $-\beta_5$, $\beta_3$, $\beta_4$, $\beta_1+\beta_5$ and $(-\beta_1+\beta_2-\beta_6)/2$  form another basis of $3A_2$ and the corresponding pullback of $\Theta_{3A_2}$ is $\Theta_{L_6}^{(2)}$.
\end{proof}

Similar to Remark \ref{rem:anotherconstruction}, there is another construction of the above basis in the context of vector-valued modular forms. In a similar way, we can demonstrate the following two lemmas and corollary.

\begin{lemma}
The vector-valued modular form corresponding to $\Theta_{L_6}$ is invariant under the orthogonal group $\Orth(D(L_6))$ up to a character of order $2$ and we have $\Theta_{L_6}=\Theta_{L_6}^{(1)}-\Theta_{L_6}^{(2)}$.
\end{lemma}

\begin{lemma}\label{lem:full3}
The function $\Grit(\Theta_{L_6})$ is a modular form of weight $3$ for $\Orth^+(M_{2,8})$ with a character of order $2$.
\end{lemma}

\begin{corollary}
We have the following equality
$$
\Borch\left(-\frac{\Theta_{L_6}|T_{-}(2)}{\Theta_{L_6}}  \right)= \Borch\left(-\frac{\Theta_{L_6}^{(1)}|T_{-}(2)}{\Theta_{L_6}^{(1)}}  \right) - \Borch\left(-\frac{\Theta_{L_6}^{(2)}|T_{-}(2)}{\Theta_{L_6}^{(2)}}  \right).
$$
\end{corollary}

The proof of  Theorem \ref{th:wt3} is now immediate.

\begin{proof}[Proof of Theorem \ref{th:wt3}]
By taking $\mathfrak{z}=z(b_2-b_4)\beta_1+ z(b_3+b_4)\beta_2+zb_5\beta_3-zb_6\beta_4+z(b_1-b_4)\beta_5+zb_4\beta_6$ in the identity of Theorem \ref{th:wt3-6}, we prove the theorem.
\end{proof}

\begin{corollary}
Let $\mathbf{a}=(a_1, a_2, a_3, a_4)\in \ZZ^4$. We define two theta blocks
\begin{align*}
\varphi^{(1)}_{3,\mathbf{b}}&=\eta^{-3}\vartheta_{b_5}
\vartheta_{b_6}\vartheta_{b_5+b_6}\vartheta_{b_3+2b_4}\vartheta_{b_3}
\vartheta_{b_2-b_4}\vartheta_{b_2+b_3+b_4}\vartheta_{b_2-b_1}\vartheta_{b_2+b_3-b_1},\\
\varphi^{(2)}_{3,\mathbf{b}}&=\eta^{-3}\vartheta_{b_5}
\vartheta_{b_6}\vartheta_{b_5+b_6}\vartheta_{b_3+2b_4}\vartheta_{b_3}
\vartheta_{b_2-b_1-b_4}\vartheta_{b_2+b_3+b_4-b_1}\vartheta_{b_2+b_3}\vartheta_{b_2}.
\end{align*}
For any $\mathbf{b}\in\ZZ^6$ such that none of $\varphi_{3,\mathbf{b}}$, $\varphi^{(1)}_{3,\mathbf{b}}$ and $\varphi^{(2)}_{3,\mathbf{b}}$ is identically zero, we have 
$$
\varphi_{3,\mathbf{b}}=\varphi^{(1)}_{3,\mathbf{b}}-\varphi^{(2)}_{3,\mathbf{b}},
$$
$$
\Borch\left(-\frac{\varphi_{3,\mathbf{b}}|T_{-}(2)}{\varphi_{3,\mathbf{b}}}
\right)=\Borch\left(-\frac{\varphi^{(1)}_{3,\mathbf{b}}|T_{-}(2)}{\varphi^{(1)}_{3,\mathbf{b}}}
\right)-\Borch\left(-\frac{\varphi^{(2)}_{3,\mathbf{b}}|T_{-}(2)}{\varphi^{(2)}_{3,\mathbf{b}}}
\right).
$$
\end{corollary}

We take $\mathbf{b}=(1,2,1,3,1,1)$ in the above corollary. Then the index of Jacobi forms is $49$ and the corresponding theta blocks are
$$
\eta^{-3}\vartheta^4\vartheta_2\vartheta_3\vartheta_4^2\vartheta_7=\eta^{-3}\vartheta^3\vartheta_2^3\vartheta_3\vartheta_5\vartheta_7-\eta^{-3}\vartheta^5\vartheta_2^2\vartheta_6\vartheta_7.
$$

At the end of this paper, we give three remarks.

\begin{remark}
Theorem \ref{th:wt2-4} and Theorem \ref{th:wt3-6} support Conjecture 4.10 in \cite{GW18} which is a generalization of theta-block conjecture to the case of orthogonal modular forms. By \cite[Remark 3.11]{GW18}, our theorems imply that there exist the hyperbolizations of the affine Lie algebras of type $A_1\oplus B_3$ and type $2A_1\oplus B_2\oplus A_2$. We can also consider the same applications of our results as in \cite[\S 4.2 and \S 4.4]{GW18}.
\end{remark}

\begin{remark}
There are four infinite series of theta blocks of weight $2$ in \cite{GSZ19}. But we can not find suitable reflective modular forms of singular weight $2$ for the other two at present. If such modular forms exist, then we know from \cite{Dit18} that the associated lattice has non-squarefree level or can not be represented as $U\oplus U(N)\oplus L(-1)$, where $N$ is the level of the lattice and it is squarefree. Besides, when the level is not squarefree, we have no idea to pass through the second difficulty mentioned in the introduction.
\end{remark}

\begin{remark}
We summarize all reflective modular forms of singular weight classified in \cite{Sch06} whose pullbacks give paramodular forms (with trivial character) of weights $2$, $3$, $4$. There are exactly $7$ such modular forms. The pullbacks of the following five reflective modular forms of singular weight give infinite families of paramodular forms which are simultaneously Borcherds products and additive liftings, which support the theta-block conjecture
\begin{align*}
&\II_{10,2}(2_{\II}^{+2})& &U\oplus U(2)\oplus E_8\cong 2U\oplus D_8& &\text{weight 4}& &\text{(see \cite{GPY15, Gri18})} &\\
&\II_{8,2}(3^{-3})& &U\oplus U(3)\oplus E_6\cong 2U\oplus 3A_2& &\text{weight 3}& &\text{(see \cite{Gri18, GSZ19})} &\\
&\II_{6,2}(5^{+3})& &U\oplus U(5)\oplus A_4\cong 2U\oplus A_4^\vee(5)& &\text{weight 2}& &\text{(see \cite{GW18})} &\\
&\II_{8,2}(2_{\II}^{+2}3^{-3})& &U\oplus U(6)\oplus E_6\cong 2U\oplus L_6& &\text{weight 3}& &\text{(see this paper)} &\\
&\II_{6,2}(2_{\II}^{+2}5^{+3})& &U\oplus U(10)\oplus A_4\cong 2U\oplus L_4& &\text{weight 2}& &\text{(see this paper)}.&
\end{align*}

The pullbacks of the following two reflective modular forms of singular weight give infinite families of antisymmetric paramodular forms of weights $3$ and $4$ (see \cite{GW19})
\begin{align*}
&\II_{8,2}(7^{-5})& &U\oplus U(7)\oplus \text{Barnes-Craig lattice}\cong 2U\oplus A_6^\vee(7)& &\text{weight 3}&\\
&\II_{10,2}(5^{+6})& &U\oplus U(5)\oplus \text{Maass lattice}\cong 2U\oplus 2A_4^\vee(5)& &\text{weight 4}.&
\end{align*}

Besides, there are also $3$ reflective modular forms of singular weight large than $3$ on lattices of squarefree level and containing $2U$.  Two of them  for the following lattices give infinite families of antisymmetric paramodular forms of weights $8$ and $6$
\begin{align*}
&\II_{18,2}(2_{\II}^{+10})& &U\oplus U(2)\oplus \text{Barnes-Wall lattice}\cong 2U\oplus E_8(2)\oplus D_8\cong 2U\oplus D_8^\vee(2)\oplus 2D_4&\\
&\II_{14,2}(3^{-8})& &U\oplus U(3)\oplus \text{Coxeter-Todd lattice}\cong 2U\oplus E_6^\vee(3)\oplus 3A_2.&
\end{align*}
The last lattice is the even unimodular lattice $\II_{26,2}$. The associated reflective modular form was first constructed in \cite{Bor95} and it has $24$ different expansions at $24$ different 1-dimensional cusps related to $24$ classes of even positive definite unimodular lattices of rank $24$ (see \cite{Gri18}).
\end{remark}

In our proofs of theorems, it is a crucial step to show that the input of the Gritsenko lift is invariant under the orthogonal group of the discriminant form as a vector-valued modular form.  We here would like to ask the following general question.

\begin{question}
Let $F$ be a modular form with a character for the full orthogonal group of $M=2U\oplus L(-1)$. Assume that its first Fourier--Jacobi coefficient, denoted by $\varphi$, is a holomorphic Jacobi form of index $1$ with trivial character for $L$. Is $\varphi$ invariant under $\Orth(D(M))$ up to a character as a vector-valued modular form?
\end{question}

\bigskip

\noindent
\textbf{Acknowledgements} 
The author would like to thank Valery Gritsenko and Yingkun Li for helpful discussions. The author is grateful to Max Planck Institute for Mathematics in Bonn for its hospitality and financial support.

\bibliographystyle{amsplain}

\end{document}